\documentclass{amsart}
\usepackage{verbatim}

\def\C{\mathcal{C}}

\def\span{\mathrm{Lin}}
\def\RR{\mathbb{R}}
\def\NN{\mathbb{N}}
\def\FF{\mathbb{F}}
\def\kk{\RR}

\def\u{\phi}
\def\uu{\psi}
\def\ub{\boldsymbol{\phi}}
\def\uub{\boldsymbol{\psi}}
\def\f{\mathbf{f}}
\def\g{\mathbf{g}}
\def\p{\mathbf{p}}
\def\LL{\mathbf{L}}
\def\H{\mathbf{H}}
\def\ll{\mathbf{l}}
\def\h{\mathbf{h}}
\def\K{\mathbf{K}}

\newtheorem{lem}{Lemma}

\newtheorem{thm}{Theorem}
\newtheorem{cor}{Corollary}

\begin{document}
\title{Local linear dependence of linear partial differential operators}
\author{J. Cimpri\v c}
\date{\today}
\begin{abstract}
We show that any finite set of linear partial differential operators with continuous coefficients is linearly dependent if and only if it is locally linearly dependent.
It follows that the reflexive closure of any finite set of such operators is equal to its linear span. The last statement can be rephrased as a weak nullstellensatz for
linear partial differential operators.
\end{abstract}

\maketitle

\section{Introduction}

Linear operators $L_1,\ldots,L_r$ from a vector space $V$ to a vector space $W$ are \textit{locally (or directionally) linearly dependent} if vectors $L_1 v,\ldots,L_r v$ are linearly dependent for every $v \in V$.
This notion was introduced in \cite{bs}. Linear dependence of $L_1,\ldots,L_r$ always implies local linear dependence but the converse can fail even for finite-dimensional $V$ and $W$. 
The aim of this short note is to show that for linear partial differential operators with continuous coefficients local linear dependence implies linear dependence. More precisely, we will prove the following result:

\begin{thm}
\label{thmmain}
Let $d$ be an integer, $U$ an open subset of $\RR^d$,
$\C(U)$ the vector space of all continuous real functions on $U$ and
$\C^{(\infty)}_c(U)$ the vector space of all infinitely differentiable real functions with compact support on $U$.

For every linear partial differential operators $L_1,\ldots,L_r$ with coefficients from $\C(U)$ the following are equivalent:
\begin{enumerate}
\item For every $\u \in \C^{(\infty)}_c(U)$, $L_1 \u,\ldots,L_r \u$ are linearly dependent in $\C(U)$.
\item $L_1,\ldots,L_r$ are linearly dependent.
\end{enumerate}
We can also replace the vector space $\C^{(\infty)}_c(U)$ with the  vector space  $\RR[x_1,\ldots,x_d]$.
\end{thm}

In Section \ref{sec3} we will prove a generalization of Theorem \ref{thmmain} to matrices of linear partial differential operators; see Theorem \ref{vecthmmain}.
An important ingredient of the proof is multivariate Hermite interpolation which will be discussed in Section \ref{sec2}. 

Corollary \ref{weaknsatz} is a simple corollary of Theorem \ref{thmmain}.
In Section \ref{sec4} we will prove its generalization 
to matrices of linear partial differential operators; see Theorem \ref{vecweaknsatz}.

\begin{cor}
\label{weaknsatz}
Let $L_1,\ldots,L_r$ and $L$ be linear partial differential operators with coefficients from $\C(U)$ where $U$ is an open subset of $\RR^d$.
The following are equivalent:
\begin{enumerate}
\item For every $\u,\uu \in \C^{(\infty)}_c(U)$ such that $\langle L_1 \u, \uu \rangle=\ldots=\langle L_r \u,\uu \rangle=0$
we have that $\langle L \u, \uu \rangle=0$. (Here $\C_c(U)$ has standard inner product $\langle f,g \rangle=\int_U fg$.)
\item For every $\u \in \C^{(\infty)}_c(U)$ we have that $L \u \in \span\{L_1 \u,\ldots,L_r \u\}$.
\item $L \in \span\{L_1,\ldots,L_r\}$.
\end{enumerate}
\end{cor}

Corollary \ref{weaknsatz} can be considered as a weak version of the Nullstellensatz for linear partial differential operators from \cite{c1}
which will also be discussed in Section \ref{sec4}.

Another motivation for studying this topic comes from Free Real Algebraic Geometry where similar results already exist.
Details will be given in Section \ref{sec5}.

\section{Auxiliary results}
\label{sec2}

Hermite interpolation is an extension of Lagrange interpolation. For given lists $(z_q,c_{q,0},c_{q,1},\ldots,c_{q,k_q})$, $q=1,\ldots,m$, of real numbers
we can find a real univariate polynomial $P$ such that $P^{(k)}(z_q)=c_{q,k}$ for every $q=1,\ldots,m$, $k=0,1,\ldots,k_q$. 
Multivariate Hermite interpolation is an extension of this to several variables.

\begin{thm}
\label{lorenz19}
Let nonnegative integers $m$ and $k_q, q=1,\ldots,m$, be given. 
For every points $z_q \in \RR^d$, $q=1,\ldots,m$, and every values $c_{q,\alpha}$ with $q=1,\ldots,m$ and $\alpha \in \NN_0^d$, $|\alpha| \le k_q$ 
there is a polynomial $P \in \Pi^d:=\RR[x_1,\ldots,x_d]$ such that for every $q$ and $\alpha$
$$P^{(\alpha)}(z_q):=\frac{\partial^{\vert \alpha \vert} P}{(\partial x_1)^{\alpha_1} \cdots (\partial x_d)^{\alpha_d}}(z_q)=c_{q,\alpha}.$$ 
\end{thm}

It can also be shown that $P$ can be chosen so that its total degree is less than $\sum_{q=1}^m (k_q+1)$, see \cite[Theorem 19]{lor}, but we will not need this in the sequel.
%Let $n$ be such that $n+1 \ge \sum_{q=1}^m (k_q+1)$ and let $\Pi_n^d$ be the set of all polynomials in $d$ variables of total degree $\le n$. 
%Here $\Pi^d=\RR[x_1,\ldots,x_d]$ and $\Pi_n^d$ is a subset of $\Pi^d$ consisting of all polynomials of total degree $\le n$.

For the sake of completeness we provide a short algebraic proof of Theorem \ref{lorenz19} which is based on the Chinese Remainder Theorem.

\begin{proof}
For each point $a=(a_1,\ldots,a_d)$ write $I_a$ for the ideal in $\Pi^d$ generated by $x_1-a_1,\ldots,x_d-a_d$.
Let us show that for any points $a \ne b$ and any $k,l \in \NN$, the ideals $I_a^k$ and $I_b^l$ are coprime.
Pick $i$ such that $a_i \ne b_i$ and note that $(x_i-a_i)^k$ and $(x_i-b_i)^l$ are relatively prime in $\RR[x_i]$.
It follows that $1=p(x_i) (x_i-a_i)^k+q(x_i) (x_i-b_i)^l \in I_a^k+I_b^l$ for suitable $p,q$.
Let $k_q$ and $z_q$ for $q=1,\ldots,m$ be as in the formulation of the theorem. We have just proved that 
the ideals $I_{z_q}^{k_q+1}$, $q=1,\ldots,m$, are pairwise coprime. By the Chinese Remainder Theorem
\cite[Proposition 1.10]{atiyah}, the canonical mapping $\Pi^d \to \prod_{q=1}^m \Pi^d/I_{z_q}^{k_q+1}$ is onto
(and its kernel is $\prod_{q=1}^m I_{z_q}^{k_q+1}=\bigcap_{q=1}^m I_{z_q}^{k_q+1}$). Therefore, there is
a polynomial $P \in \Pi^d$ such that for every $q=1,\ldots,m$,
$$P \equiv \sum\limits_{\alpha \in \NN_0^d, \vert \alpha \vert \le k_q} c_{q,\alpha} (x-z_q)^\alpha \mod I_{z_q}^{k_q+1} $$
\end{proof}

The following observation is well-known.

\begin{lem}
\label{points}
Functions $f_1,\ldots,f_k \in \C(U)$ are linearly independent iff there exist points $z_1,\ldots,z_k \in U$ such that $\det [f_i(z_j)]_{i,j=1,\ldots,k} \ne 0$.
\end{lem}

We will need a generalization of Lemma \ref{points} to $n$-tuples of functions.

\begin{lem}
\label{vecpoints}
For every $\f_1,\ldots,\f_r \in \C(U)^n$ the following are equivalent:
\begin{enumerate}
\item $\f_1,\ldots,\f_r$ are linearly independent.
\item There exist points $z_1,\ldots,z_r \in U$ and indices $i_1,\ldots,i_r \in \{1,\ldots,n\}$ such that
$\det[ f_{l,i_k} (z_k) ]_{k,l=1,\ldots,r} \ne 0$.
\item There exist points $z_1,\ldots,z_r \in U$ such the block matrix $[\f_l(z_k]]_{k,l=1,\ldots,r}$ has maximal rank.
\end{enumerate}
\end{lem}

\begin{proof}
Clearly, (2) implies (3) and (3) implies (1). We will prove that (1) implies (2) by induction on $r$.

If (1) is true for $r=1$ then $\f_1 \ne 0$. Thus there exists $z_1 \in U$ such that $\f_1(z_1) \ne 0$. This implies (2).

If (1) is true for $\f_1,\ldots,\f_r$ then it is also true for $\f_1,\ldots,\f_{r-1}$. By induction hypothesis, there exist points $z_1,\ldots,z_{r-1} \in U$ and 
indices $i_1,\ldots,i_{r-1} \in \{1,\ldots,n\}$ such that $$\det[ f_{l,i_k} (z_k) ]_{k,l=1,\ldots,r-1} \ne 0.$$ 
Consider the matrices
$$A_{i,z}:=\left[ \begin{array}{ccc} f_{1,i_1}(z_1) & \ldots & f_{r,i_1}(z_1) \\ \vdots & & \vdots \\ 
f_{1,i_{r-1}}(z_{r-1}) & \ldots & f_{r,i_{r-1}}(z_{r-1}) \\ f_{1,i}(z) & \ldots & f_{r,i}(z) \end{array} \right].$$
If (2) is false then $\det A_{i,z}=0$ for all $i$ and all $z$. By expanding $\det A_{i,z}$ along the last row, we get
the $i$-th row of the equation $$c_1 \f_1(z)+\ldots+c_r \f_r(z)=0$$
where $c_r = \det[ f_{l,i_k} (z_k) ]_{k,l=1,\ldots,r-1} \ne 0$. This contradicts (1).
 Thus (2) is true.
\end{proof}

\section{Matrices of linear partial differential operators}
\label{sec3}

The aim of this section is to prove a generalization of Theorem \ref{thmmain} to matrices of linear partial differential operators.

Let $\LL_1,\ldots,\LL_r$ be $m \times n$ matrices of partial differential operators with continuous coefficients. 
Let $s$ be the maximum of total degrees of all entries of all matrices and let $I_s := \{ \alpha \in \NN_0^d \colon \mid \alpha \mid \le s\}$.
Then we can write 
$$\LL_i=\sum_{\alpha \in I_s} \mathbf{P}_{i,\alpha} D^{\alpha}, \quad i=1,\ldots,r$$ where $\mathbf{P}_{i,\alpha} \in \C(U)^{m \times n}$ 
and $D^\alpha=\frac{\partial^{\vert \alpha \vert} P}{(\partial x_1)^{\alpha_1} \cdots (\partial x_d)^{\alpha_d}}$.

\begin{thm}
\label{vecthmmain}
Let $\LL_1,\ldots,\LL_r$ be as above. The following are equivalent:
\begin{enumerate}
\item For every $\ub \in \C_c^{(\infty)}(U)^n$ we have that $\LL_1 \ub,\ldots,\LL_r \ub$ are linearly dependent.
\item $\LL_1,\ldots,\LL_r$ are linearly dependent.
\end{enumerate}
\end{thm}

\begin{proof}
Clearly, (2) implies (1). To show that (1) implies (2) we have to pick first a suitable $\ub$.
Write $$\mathbf{P}_{i,\alpha}=\left[ \p_{i,\alpha,1},\ldots,\p_{i,\alpha,n} \right]$$ for each $i$ and $\alpha$.
For each $\alpha$ and $k$ write $V_{\alpha,k}=\span\{\p_{1,\alpha,k},\ldots,\p_{r,\alpha,k}\}$. Pick a subset
$J_{\alpha,k}$ of $\{1,\ldots,r\}$ such that $\{\p_{j,\alpha,k} \mid j \in J_{\alpha,k}\}$ is a basis of $V_{\alpha,k}$.
By Lemma \ref{vecpoints} there exist points $z_{i,\alpha,k} \in \RR^d$, $i \in J_{\alpha,k}$ such that the matrix
$$Q_{\alpha,k}:=[\p_{j,\alpha,k}(z_{i,\alpha,k})]_{i,j \in J_{\alpha,k}}$$ has maximal rank.

We can choose $z_{i,\alpha,k}$ so that $z_{i,\alpha,k} \ne z_{j,\beta,l}$ if $(\alpha,k) \ne (\beta,l)$.
This follows from a simple observation that elements of $\C(U)^n$ are linearly independent if and only if 
their restrictions to $U \setminus Z$, where $Z$ is a finite set, are linearly independent. 
We can therefore assume that all zeroes $z_{i,\alpha,k}$ are pairwise distinct and indexed by the set
$$\Lambda:=\{(i,\alpha,k) \mid k=1,\ldots,n, \, \alpha \in I_s,\, i \in J_{\alpha,k}\}.$$

By Theorem \ref{lorenz19}, there exist polynomials $\u_1,\ldots,\u_n \in \RR[x_1,\ldots,x_d]$ such that 
$$\u_k^{(\alpha)}(z_{j,\beta,l})=\delta_{k,l}\delta_{\alpha,\beta}=\left\{ \begin{array}{cc} 1 & (\alpha,k)=(\beta,l) \\ 
0 & (\alpha,k) \ne (\beta,l) \end{array} \right.$$
for each $k=1,\ldots,n$, each $\alpha \in I_s$ and each $(j,\beta,l) \in \Lambda$.

We can chosse two balls around each $z_{i,\alpha,k}$, $(i,\alpha,k) \in \Lambda$, both contained in $U$ 
and an infinitely differentiable function $h$ which is equal to $1$ on the union of smaller balls and equal to $0$ 
oustide the union of the bigger balls. Replacing $\u_1,\ldots,\u_n$ with $\u_1 h,\ldots,\u_n h$ we may assume that
$\u_1,\ldots,\u_n \in \C_c^{(\infty)}(U)$ and still satisfy $\u_k^{(\alpha)}(z_{j,\beta,l})=\delta_{k,l}\delta_{\alpha,\beta}$.

Consider the matrix 
$$Q:= [ \p_{i, \alpha,k} (z_{j,\beta,l}) \u_k^{(\alpha)}(z_{j,\beta,l}) ]_{(i,\alpha,k),(j,\beta,l) \in \Lambda}.  $$
By the choice of $\u_1,\ldots,\u_n$ we have that 
$$Q= [ \p_{i, \alpha,k} (z_{j,\beta,l}) \delta_{k,l}\delta_{\alpha,\beta}) ]_{(i,\alpha,k),(j,\beta,l) \in \Lambda}=\bigoplus\limits_{(\alpha,k)} Q_{\alpha,k}$$ 
where $Q_{\alpha,k}$ are defined above. Since $Q_{\alpha,k}$ have maximal rank, so has $Q$. By Lemma \ref{vecpoints} it follows that the vectors 
$$\p_{i, \alpha,k}  \u_k^{(\alpha)} \in \C(U)^n, \quad (i,\alpha,k) \in \Lambda$$
are linearly independent.

Write $\ub=[\u_1,\ldots,\u_n]^T$ and pick $c_1,\ldots,c_r \in \RR$, at least one nonzero, such that
$$ 0=c_1 \LL_1 \ub+\ldots+c_r \LL_r \ub.$$
It follows that $$0=\sum_{i=1}^r c_i \, \LL_i \ub=\sum_{i=1}^r c_i \sum_{\alpha \in I_s} \mathbf{P}_{i,\alpha} D^\alpha \ub
= \sum_{i=1}^r c_i \sum_{\alpha \in I_s} \sum_{k=1}^n \p_{i,\alpha,k} \u_k ^{(\alpha)}=$$
$$=\sum_{\alpha \in I_s} \sum_{k=1}^n (\sum_{i=1}^r c_i\p_{i,\alpha,k}) \u_k ^{(\alpha)}=
\sum_{\alpha \in I_s} \sum_{k=1}^n (\sum_{i \in J_{\alpha,k}} d_{i,\alpha,k} \p_{i,\alpha,k}) \u_k ^{(\alpha)}=$$
$$=\sum_{(i,\alpha,k) \in \Lambda}d_{i,\alpha,k} \p_{i,\alpha,k} \u_k ^{(\alpha)}.$$
Since $\p_{i, \alpha,k}  \u_k^{(\alpha)}$ are linearly independent, it follows that $d_{i,\alpha,k}=0$ for all $(i,\alpha,k)$.
Therefore $$\sum_{i=1}^r c_i\p_{i,\alpha,k}=\sum_{i \in J_{\alpha,k}} d_{i,\alpha,k} \p_{i,\alpha,k}=0$$
for all $(\alpha,k)$ which implies that
$$\sum_{i=1}^r c_i \LL_i=\sum_{i=1}^r c_i \sum_{\alpha \in I_s} \mathbf{P}_{i,\alpha} D^\alpha
= \sum_{\alpha \in I_s} \sum_{i=1}^r c_i \mathbf{P}_{i,\alpha} D^\alpha=$$
$$=\sum_{\alpha \in I_s} \sum_{i=1}^r c_i[\p_{i,\alpha,1},\ldots,\p_{i,\alpha,n}] D^\alpha
=\sum_{\alpha \in I_s} [\sum_{i=1}^r c_i\p_{i,\alpha,1},\ldots,\sum_{i=1}^r c_i\p_{i,\alpha,n}] D^\alpha=0.$$
\end{proof}

\section{A weak nullstellensatz for matrices of differential operators}
\label{sec4}

The aim of this section is to prove Theorem \ref{vecweaknsatz} which is a generalization of Corollary \ref{weaknsatz} to matrices of partial differential operators.
Theorem \ref{vecweaknsatz} can be considered as a weak analogue of Theorem \ref{vecstrongnsatz} below. 
%We will also discuss analogues of Theorems \ref{vecweaknsatz} and \ref{vecstrongnsatz} in Free Real Algebraic Geometry.

\begin{thm}
\label{vecweaknsatz}
Let $\LL_1,\ldots,\LL_r$ and $\LL$ be $m \times n$ matrices of linear partial differential operators with coefficients from $\C(U)$ where $U$ is an open subset of $\RR^d$.
The following are equivalent:
\begin{enumerate}
\item For every $\ub \in \C^{(\infty)}_c(U)^n$ and every  $\uub \in \C^{(\infty)}_c(U)^m$ such that $\langle \LL_1 \ub, \uub \rangle=\ldots=\langle \LL_r \ub,\uub \rangle=0$
we have that $\langle \LL \ub, \uub \rangle=0$. (Here $\C_c(U)^m$ has standard inner product $\langle \f,\g \rangle=\sum_{i=1}^m \int_U f_i g_i$.)
\item For every $\ub \in \C^{(\infty)}_c(U)^n$ we have that $\LL \ub \in \span\{\LL_1 \ub,\ldots,\LL_r \ub\}$.
\item $\LL \in \span\{\LL_1,\ldots,\LL_r\}$.
\end{enumerate}
\end{thm}

\begin{proof}
Clearly, (3) implies (2) and (2) implies (1). By Theorem \ref{thmmain}, (2) implies (3). Namely, without loss of generality,
we can assume that $\LL_1,\ldots,\LL_r$ are linearly independent. By (2), $\LL_1,\ldots,\LL_r,\LL$ are locally linearly dependent.

To show that (1) implies (2), pick $\ub \in \C(U)^n$ and note that the sets $U_{\ub}:=\span\{\LL_1 \ub,\ldots,\LL_r \ub\}$ and $V_{\ub}:=\span\{\LL \ub\}$
are finite-dimensional subspaces of $\C_c(U)^m$.
By (1), we have that $\uub \in U_{\ub}^\perp$ implies $\uub \in V_{\ub}^\perp$, i.e. $U_{\ub}^\perp \subseteq V_{\ub}^\perp$.
It follows that $U_{\ub}^{\perp \perp} \subseteq V_{\ub}^{\perp \perp}$. By Lemma \ref{lemperp} below, it follows that $V_{\ub} \subseteq U_{\ub}$.
\end{proof}

In the proof we used the following simple and well-known observation:

\begin{lem}
\label{lemperp}
If $V$ is an inner product space and $W$ is a finite-dimensional subspace of $V$ then $W^{\perp \perp}=W$.
\end{lem}

\begin{proof}
Let $e_1,\ldots,e_n$ be an orthonormal basis of $W$. For every element $v \in V$ write
$v'=\sum_{i=1}^n \langle v,e_i \rangle e_i$ and $v''=v-v'$. Clearly. $v' \in W$, $v'' \in W^\perp$ and $v=v'+v''$.
It follows that $V= W \oplus W^\perp$. Finally, if $v=v'+v'' \in W^{\perp\perp}$, then $\langle v, v'' \rangle =0$
which implies that $\langle v'', v'' \rangle =0$. Therefore $v=v' \in W$.
\end{proof}

Theorem \ref{vecstrongnsatz} can be considered as a strong nullstellensatz for matrices of linear partial differential operators:

\begin{thm}
\label{vecstrongnsatz}
Let $U$ be a nonempty open subset of $\RR^d$ and let $\LL_1,\ldots,\LL_r$ and $\LL$ be $m \times n$ matrices of linear partial differential operators with polynomial coefficients. The following are equivalent:
\begin{enumerate}
\item[(1)] Every $n$-tuple of convergent power series $\ub$ around every point of $U$ which satisfies $\LL_1 \ub=\ldots=\LL_r \ub=0$ also satisfies $\LL \ub=0$.
\item[(2)] There exists a nonzero polynomial $w$ and $m \times m$ matrices of linear partial differential operators $\H_1,\ldots,\H_r$ with polynomial coefficients such that $w \LL=\H_1 \LL_1 + \ldots +\H_r \LL_r$. 
\end{enumerate}
\end{thm}

\begin{proof}
The $m=1$ case of Theorem \ref{vecstrongnsatz} was proved in \cite{c1}. Write
$$\K=\left[ \begin{array}{c} \LL_1 \\ \vdots \\ \LL_r \end{array} \right] \quad \text{and} \quad \LL=\left[ \begin{array}{c} \ll_1 \\ \vdots \\ \ll_m \end{array} \right].$$
By assumption (1), $\K \ub=0$ implies $\ll_i \ub=0$ for every $i=1,\ldots,m$. By the $m=1$ case there exist row matrices $\h_1,\ldots,\h_m$ of linear partial 
differential operators and a nonzero polynomial $w$ such that $w \ll_i=\h_i K$  for every $i=1,\ldots,m$. It follows that
$$w \LL= \left[ \begin{array}{c} \h_1 K \\ \vdots \\ \h_m K \end{array} \right] = \left[ \begin{array}{c} \h_1 \\ \vdots \\ \h_m \end{array} \right] K =
\left[ \begin{array}{ccc} \H_1 & \ldots & \H_r \end{array} \right] K = \H_1 \LL_1 + \ldots +\H_r \LL_r.$$  
\end{proof}

\section{Free polynomials and Matrix Polynomials}
\label{sec5}

Let us survey analogues of our main results for free polynomials and matrix polynomials.

Theorem \ref{freemain} is an analogue of Theorem \ref{thmmain} for free polynomials.
It is a special case of \cite[Theorem 3.1 and Theorem 3.7]{bk} (applied to $\mathcal{A}=\bigcup_{n=1}^\infty M_n(\FF)$.)
See also \cite{chsy}. It would be interesting to have versions of these results for matrices over $\FF \langle x_1,\ldots,x_d \rangle$.

\begin{thm}
\label{freemain}
Let $\FF$ be a field. For any $p_1,\ldots,p_m$ from the free algebra $\FF \langle x_1,\ldots,x_d \rangle$ the following are equivalent.
\begin{enumerate}
\item For every $n \in \NN$, every $A_1,\ldots,A_d \in \FF^{n \times n}$ and every $v \in \FF^n$, we have that $p_i(A_1,\ldots,A_d)$v,
\ldots, $p_i(A_1,\ldots,A_d)v$ are linearly dependent.
\item For every $n \in \NN$ and every $A_1,\ldots,A_d \in \FF^{n \times n}$, we have that $p_i(A_1,\ldots,A_d)$,
\ldots, $p_i(A_1,\ldots,A_d)$ are linearly dependent.
\item $p_1,\ldots,p_m$ are linearly dependent.
\end{enumerate}
\end{thm}

Corollary \ref{freeweaknsatz} is an analogue of Corollary \ref{weaknsatz} for free polynomials. It follows from Theorem \ref{freemain} 
by the same argument as in the proof of Theorem \ref{vecweaknsatz}.

\begin{cor}
\label{freeweaknsatz}
Let $\FF$ be a field. For any $p_1,\ldots,p_m,q \in \FF \langle x_1,\ldots,x_d \rangle$ the following are equivalent.
\begin{enumerate}
\item For every $n \in \NN$, every $A_1,\ldots,A_d \in \FF^{n \times n}$ and every $u,v \in \FF^d$ such that $\langle p_i(A_1,\ldots,A_d)u,v \rangle=0$ for $i=1,\ldots, m$
we have  $\langle q(A_1,\ldots,A_d)u,v \rangle=0$.
\item For every $n \in \NN$, every $A_1,\ldots,A_d \in \FF^{n \times n}$ and every $u \in \FF^d$ we have that $q(A_1,\ldots,A_d)u \in \span \{ p_i(A_1,\ldots,A_d)u \mid i=1,\ldots, m\}$.  
\item For every $n \in \NN$ and every $A_1,\ldots,A_d \in \FF^{n \times n}$, we have that $q(A_1,\ldots,A_d) \in \span \{ p_i(A_1,\ldots,A_d) \mid i=1,\ldots, m\}$.  
\item $q \in \span \{ p_i \mid i=1,\ldots, m\}$.  
\end{enumerate}
\end{cor}

A variant of Theorem \ref{vecstrongnsatz} for free polynomials is proved in \cite{c2} and extended to matrices of free polynomials in \cite{nelson}.
It works only for real and complex coefficients and it requires the notion of a real radical of a left ideal.

The precise analogue of Theorem \ref{vecweaknsatz} for matrices of usual polynomials fails by \cite[Example 3.1]{akm}.
(This example uses only constant polynomials but it cannot be adapted to matrices of free polynomials because evaluations are different.) 
For given $P_1,\ldots,P_r \in M_n(\FF[x_1,\ldots,x_d])$ it would be interesting to have an algebraic 
characterization of the set of all $Q \in M_n(\FF[x_1,\ldots,x_d])$ for which $Q(a)v \in \span\{P_1(a)v,\ldots,P_r(a)v\}$
for every $a \in \FF^d$ and every $v \in \FF^n$.

A variant of Theorem \ref{vecstrongnsatz} for matrices of usual polynomials was proved in \cite[Theorem 3]{c3}. 
It also requires real or complex coefficients and the notion of a real radical of a left ideal.

\end{document}